\documentclass{article}
\usepackage{kms}

\title{\textbf{The characteristic polynomial of an algebra and representations}}
\author{
	Rajesh S.~Kulkarni 
		\footnote{Michigan State University, East Lansing, Michigan. {\sf kulkarni@math.msu.edu}},
	Yusuf Mustopa 
		\footnote{Tufts University, Medford, Massachusetts. {\sf Yusuf.Mustopa@tufts.edu}}, and 
	Ian Shipman 
		\footnote{University of Utah, Salt Lake City, Utah. {\sf ian.shipman@gmail.com}}
}

\begin{document}
\maketitle

Suppose that $\bk$ is a field and let $A$ be a finite dimensional, associative, unital $\bk$-algebra.  Often one is interested in studying the finite-dimensional representations of $A$.  Of course, a finite dimensional representation of $A$ is simply a finite dimensional $\bk$-vector space $M$ and a $\bk$-algebra homomorphism $A \to \End_\bk(M)$.  In this article we will not consider representations of algebras, but rather how to determine if a $\bk$-linear map $\phi:A \to \End_\bk(M)$ is actually a homomorphism.  We restrict our attention to the case where $A$ is a product of field extensions of $\bk$.  If $\phi:A \to \End_\bk(M)$ is a representation then certainly, if $a \in A$ satisfies $a^m = 1$ then $\phi(a)^m = \id$ as well.  Our first Theorem is a remarkable converse to this elementary observation.

\begin{customthm}{A}\label{thm:roots}
Suppose that $A = \bk^d$ and $\phi:A \to \End_\bk(M)$ is a $\bk$-linear map.  Let $n > 2$ be a natural number and assume that $\bk$ has $n$ primitive $n^{th}$ roots of unity.  If $\phi(1_{A}) = \id$ and for each $a \in A$ such that $a^n = 1_{A}$ we have $\phi(a)^n = \id,$ then $\phi$ is an algebra homomorphism.
\end{customthm}

Consider the regular representation $\mu_L:A \to \End_\bk(A)$ of $A$ on itself by left multiplication.  For $a \in A$, let $\chi_a(t)$ and $\ol{\chi}_a(t)$ be the characteristic and minimal polynomials of $\mu_L(a)$, respectively.  We note that $\chi_a(a) = \ol{\chi}_a(a) = 0$ in $A$.  Therefore if $M$ is a finite dimensional left $A$ module with structure map $\phi:A \to \End_\bk(M)$ then $\chi_a(\phi(a)) = \ol{\chi}_a(\phi(a)) = 0$ in $\End_\bk(M)$.  The notion of assigning a characteristic polynomial to each element of an algebra and considering representations which are compatible with this assignment has appeared in \cite{Pr}.  This idea has been applied to some problems in noncommutative geometry as well \cite{L}.  However, as far as we know the following related notion is new.

\begin{definition}  Suppose that $\phi:A \to \End_\bk(M)$ is a $\bk$-linear map, where $M$ is a finite dimensional $\bk$-vector space.  We say that $\phi$ is a \emph{characteristic morphism} if $\chi_a(\phi(a)) = 0$ for all $a \in A$.  We say that $\phi$ is \emph{minimal-characteristic} if, moreover, $\ol{\chi}_a(\phi(a)) = 0$ for all $a \in A$.
\end{definition}

It is natural to ask whether or not the notions of characteristic morphism and minimal characteristic morphism are weaker than the notion of algebra morphism.  Let us address minimal-characteristic morphisms first.

\begin{cor*}
Assume that $A = \bk^d$ and that $\bk$ has a full set of $d^{th}$ roots of unity.  Then a minimal-characteristic morphism $\phi:A \to \End_\bk(M)$ is an algebra morphism.
\end{cor*}
\begin{proof}
First note that $\ol{\chi}_1(t) = t -1$.  Hence $\phi(1) = \id$.  Furthermore if $a \in A$ satisfies $a^d = 1$ then $\ol{\chi}_a(t)$ divides $t^d -1$.  Therefore, $\phi(a)^d = \id$.  Hence, if $d > 2$ then Theorem \ref{thm:roots} implies that $\phi$ is an algebra morphism.  We leave the cases $d = 1,2$ for the reader.
\end{proof}

\begin{example}\label{exm1}
Let $a,b \in \bk$ be such that $a + b \neq 0$.  Then the map $\phi:\bk^{\times 2} \to \Mat_2(\bk)$ given by
\[ \phi(e_1) = \begin{pmatrix} 1 & a \\ 0 & 0 \end{pmatrix}, \quad \phi(e_2) = \begin{pmatrix} 0 & b \\ 0 & 1 \end{pmatrix} \]
is a characteristic morphism that is not a representation.  
\end{example}

Characteristic morphisms form a category in a natural way.  Any linear map $\phi:A \to \End_\bk(M)$ endows $M$ with the structure of a $\rmT(A)$ module, where $\rmT(A)$ denotes the tensor algebra on $A$.  We can view the characteristic polynomial of elements of $A$ as a homogeneous form $\chi(t) \in \Sym^\bt(A^\vee)[t]$ of degree $d$, monic in $t$.  Pappacena \cite{P} associates to such a form an algebra
\[ C(A) = \frac{\rmT(A)}{\langle \chi_a(a) : a \in A \rangle}, \]
where if $\chi_a(t) = \sum_{i=0}^d{c_i(a)t^i}$ then 
\[ \chi_a(a) := \sum_{i=0}^d{ c_i(a) a^{\tensor i} } \in \rmT(A). \]
Clearly, the action map $\phi:A \to \End_\bk(M)$ of a $\rmT(A)$-module $M$ is a characteristic morphism if and only if the action of $\rmT(A)$ factors through $C(A)$.  We declare the category of characteristic morphisms to be the category of finite-dimensional $C(A)$-modules.  So we have a notion of irreducible characteristic morphism.  The characteristic morphism constructed in Example \ref{exm1} is not irreducible, being an extension of two irreducible characteristic morphisms.  However, every irreducible characteristic morphism $\bk^{\times 2} \to \End_\bk(M)$ is actually an algebra morphism.  The following example shows this is not always the case.

\begin{example}
The linear map $\phi : \bk^{\times 3} \to \Mat_3(\bk)$ defined by
\[ e_1 \mapsto \frac{1}{2}\begin{pmatrix} 0 & 1 & 1 \\ 0 & 1 & -1 \\ 0 & -1 & 1  \end{pmatrix},
 \quad e_2 \mapsto \frac{1}{2}\begin{pmatrix} 1 & 0 & -1 \\ 1 & 0 & 1 \\ -1 & 0 & 1 \end{pmatrix},
 \quad e_3 \mapsto \frac{1}{2}\begin{pmatrix} 1 & -1 & 0 \\ -1 & 1 & 0 \\ 1 & 1 & 0 \end{pmatrix} \]
is an irreducible characteristic morphism, but not an algebra morphism; while $e_1^{2} = e_1,$ it can be checked that $\phi(e_{1})^{2} \neq e_1.$
\end{example}

Given a linear map $\phi:A \to \End_\bk(M)$, let $T_\phi \in A^\vee \tensor \End_\bk(M)$ be the element that corresponds to $\phi$ under the isomorphism $\Hom_\bk(A,\End_\bk(M)) \cong A^\vee \tensor \End_\bk(M)$.  We view $T_\phi$ as an element of $\Sym^\bt(A^\vee) \tensor \End_\bk(M)$.  The equation $\chi_a(\phi(a))=0$ for all $a \in A$ holds if and only if $\chi_A(T_\phi) = 0$ in $\Sym^\bt(A^\vee)\tensor \End_\bk(M)$.  We can just as easily view $T_\phi$ as an element of $\rmT(A^\vee)\tensor \End_\bk(M)$.  Moreover, we can lift $\chi$ from $\Sym^\bt(A^\vee)[t]$ to $\rmT(A^\vee) \ast \bk[t]$ by the na\"{i}ve symmetrization map $\Sym^\bt(A^\vee)[t] \to \rmT(A^\vee)\ast \bk[t]$.  

\begin{customthm}{B}\label{thm:noncommutative}
  Assume that $\chr(\bk)$ is either $0$ or greater than $d$.  Let $A = \bk^{\times d}$ and $\phi:A \to \End_\bk(M)$ a $\bk$-linear map.  The map $\phi$ factors through a homomorphism $A\to \End_\bk(M)$ if and only if $\chi(T) = 0$ in $\rmT(A^\vee) \tensor_\bk B$.
\end{customthm}

\textbf{Acknowledgments.} We would like to thank Ted Chinburg and Zongzhu Lin for interesting conversations.  I.S. was partially supported during the preparation of this paper by National Science Foundation award DMS-1204733.  R. K. was partially supported by the National Science Foundation awards DMS-1004306 and DMS-1305377.


\section*{Proofs}
We now turn to the proofs of the results in the introduction.  The proof of the Theorem \ref{thm:roots} depends on an arithmetic Lemma.
\begin{lem}\label{lem:arithmetic}
Let $\zeta \in \bk$ be a primitive $n^{th}$ root of unity.  Suppose that $a,b,c,d \in \Z$ satisfy $b,d \neq 0 \mod n$ and
\[ \frac{ \zeta^a - 1}{\zeta^b -1 } = \frac{ \zeta^c -1 }{ \zeta^d -1 }. \]
Then either:
\begin{enumerate}
\item $a \equiv b \mod n$ and $c \equiv d \mod n$, or
\item $a \equiv c \mod n$ and $b \equiv d \mod n$.
\end{enumerate}
\end{lem}
\begin{proof}
After possibly passing to a finite extension we may assume that $\bk$ admits an automorphism sending $\zeta$ to $\zeta^{-1}$.  Thus we have
\[ 
\frac{ \zeta^{-a} - 1 }{ \zeta^{-b} - 1 } = \frac{ \zeta^{-c} - 1}{\zeta^{-d} - 1}, 
\]
which we rewrite
\[
\frac{ \zeta^{-a} }{\zeta^{-b} } \cdot \frac{ 1 - \zeta^a }{ 1 - \zeta^b } = \frac{ \zeta^{-c} }{\zeta^{-d}} \cdot \frac{ 1 - \zeta^c }{1-\zeta^d}.
\]
Using our assumption we find that $\zeta^{b-a} = \zeta^{d-c}$.  Thus $b - a \equiv d - c \mod n$.  Let $e = b - a \equiv d-c \mod n$.  Then we have
\[ \frac{ \zeta^{b-e} -1}{\zeta^b -1 } = \frac{ \zeta^{d-e} -1}{\zeta^d -1 } \]
which implies that
\[ \zeta^{b-e} + \zeta^d = \zeta^{d-e} + \zeta^b. \]
Finally we see that 
\[ \zeta^d - \zeta^b = (\zeta^d - \zeta^b)\zeta^{-e} \]
Therefore either $e \equiv 0 \mod d$ so that (1) holds, or $d \equiv b$ so that (2) holds.
\end{proof}

\begin{proof}[Proof of Theorem \ref{thm:roots}]
Whether or not $\phi$ is an algebra homomorphism is stable under passage to the algebraic closure of $\bk$.  So we may assume that $\bk$ is algebraically closed.  Let $e_1,\dotsc,e_d \in A$ be a complete set of orthogonal idempotents.  Put $\alpha_i = \phi(e_i)$ and note that by hypothesis $\alpha_1 + \dotsm + \alpha_d = \id$.  Fix a primitive $n^{th}$ root of unity $\xi$.  Then $x = 1 + (\xi - 1)e_i$ satisfies $x^n = 1$.  Therefore $\phi(x)^d = \id$.  This implies that $\phi(x)$ is diagonalizable and each eigenvalue is an $n^{th}$ root of unity.  Now, since $\phi$ is linear,
\[ \alpha_i = \frac{ \phi(x) - \id }{\xi - 1 } \]
and hence $\alpha_i$ is diagonalizable as well.  Let $\lambda$ be an eigenvalue of $\alpha_i$.  Then for some $a$ we have
\[ \lambda = \frac{\xi^a - 1}{\xi - 1}. \]
Now for any $b$, $\phi( 1 + (\xi^b - 1)e_i )^d = \id$.  So we see that 
\[ 1 + \lambda (\xi^b - 1) \]
must be a root of unity for every $b$.  However, if  
\[ 1 + \lambda (\xi^b -1) = \xi^c \]
then Lemma \ref{lem:arithmetic} implies that either $a \equiv 1 \mod n$, $\lambda = 0$, or $b \equiv 1 \mod n$.  Now, $b$ is under our control and since $n \geq 3$ we can choose $b \neq 0,1 \mod n$, excluding the third case.  If $a \equiv 1 \mod n$ then $\lambda = 1$ and otherwise $\lambda = 0$.  Thus $\alpha_i$ is semisimple with eigenvalues equal to zero or one.  So $\alpha_i^2 = \alpha_i$.

Let $i \neq j$ and consider
\[ y_a = \id + (\xi^a -1) (\alpha_i +\alpha_j) \]
Clearly, $y_a^n = \id$ and thus $y_a$ is semisimple.  We compute
\[ (y_a - \id)^2 = (\xi^a-1)^2( \alpha_i \alpha_j + \alpha_j \alpha_i ) + (\xi^a - 1)(y_a - \id) \]
and deduce that 
\begin{equation}\label{eq:anticommutator} (\xi_a -1)^{-2}(y_a -\id)(y_a - \xi^a) = (\alpha_i \alpha_j + \alpha_j \alpha_i). \end{equation}
Assume that $b \neq 0 \mod n$.  Observe that $y_a - \id = \frac{\xi^a -1}{\xi^b - 1} (y_b - \id)$ and therefore, $y_a$ and $y_b$ are simultaneously diagonalizable.  Suppose that $\xi^c$ is an eigenvalue of $y_b$ unequal to 1.  Then
\[ \frac{\xi^a - 1}{\xi^b -1}(\xi^c - 1) + 1 = \xi^e\]
is an eigenvalue of $y_a$.  Since $n \geq 3$ we can assume that $a \neq b,0 \mod n$.  Then Lemma \ref{lem:arithmetic} implies that $e \equiv a \mod n$ and $b \equiv c \mod n$.  We see that the only eigenvalues of $y_b$ are 1 and $\xi^b$.

Because $y_b$ is semisimple, this means that the right side of \eqref{eq:anticommutator} vanishes.  So $\alpha_i \alpha_j = - \alpha_j \alpha_i$ for all $i,j$.  Suppose that $\alpha_i(m) = m$.  Then $\alpha_j(\alpha_i(m)) = \alpha_j(m) = - \alpha_i(\alpha_j(m))$.  Since $-1$ is not an eigenvalue of $\alpha_i$ we see that $\alpha_j(m) = 0$.  Now let $m \in M$ and write $m = m_0 + m_1$ where $\alpha_i(m_0) = 0$ and $\alpha_i(m_1) = m_1$.  Then 
\[ \alpha_i(\alpha_j(m)) = \alpha_i(\alpha_j(m_0)) = - \alpha_j(\alpha_i(m_0)) = 0. \]
Thus we see that in fact $\alpha_i \alpha_j = 0$.  So $\alpha_1,\dotsc,\alpha_d$ satisfy the defining relations of $\bk^{\times d}$ and $\phi$ is actually an algebra homomorphism.
\end{proof}

We now turn to the proof of Theorem \ref{thm:noncommutative}.  The key idea is to use the fact that the single equation $\chi(T_\phi)=0$ over the tensor algebra encodes many relations for the matrices defining $\phi$.  It is convenient to consider $\alpha_i = \phi(e_i)$, where $e_i$ is the standard basis of idempotents in $\bk^{\times d}$.  Furthermore we write $\chi_d$ for the characteristic polynomial of $\bk^{\times d}$ viewed as an element of $\bk\langle x_1,\dotsc,x_d,t\rangle$ (where $x_1,\dotsc,x_d$ is the dual basis to $e_1,\dotsc,e_d$).

\begin{lem}\label{lem:actions} 
Suppose that $\bk$ is a field with $char(\bk) > d$.  Let $\alpha_1,\dotsc,\alpha_d \in M_n(\bk)$ and put $T = x_1 \alpha_1 + \dotsc + x_d \alpha_d$.  If $T$ satisfies $\chi_d$ then
\begin{enumerate}
\item for some $i =1,\dotsc,d$, $\alpha_i$ has a 1-eigenvector, and
\item if $m \in \bk^n$ satisfies $\alpha_im = m$ then $\alpha_j m = 0$ for all $j \neq i$.
\end{enumerate}
\end{lem}
\begin{proof}
(1.) Let $S = \bk[x_1,\dotsc,x_d]$ as an $A = \bk\langle x_1,\dotsc,x_d\rangle$ module in the obvious way.  Then the image of $\chi_d$ in $\bk[x_1,\dotsc,x_d,t]$ is $p(t) =n! (t - x_1)\dotsm(t-x_d)$, where now the order of the terms does not matter.  Hence $T$ satisfies $(T-x_1)\dotsm (T - x_d) = 0$ in $M_n(S)$.  So we can view $S^n$ as an $R=\bk[x_1,\dotsc,x_d,t]/(p(t))$-module $M$.  For each $i$ consider the quotient $S_i := R/(t-x_i)$, which is isomorphic to $S$ under the natural map $S \to S_i$.  Define $M_i = M \tensor_R S_i$.  Since the map $S \to S_1 \times \dotsm \times S_d$ is an isomorphism after inverting $a = \prod_{i \neq j}(x_i - x_j)$ and $a$ is a nonzerodivisor on $M$, the natural map $M \to M_1 \oplus \dotsm \oplus M_d$ is injective.  Hence there is some $i$ such that $M_i$ has positive rank.  Consider $\bar{M} := M/(x_1,\dotsc,x_{i-1},x_i - 1,x_{i+1},\dotsc,x_d)M$ and $\bar{M}_i := M_i/(x_1,\dotsc,x_{i-1},x_i-1,x_{i+1},\dotsc,x_d)M_i$.  Now since $M_i$ (is finitely generated and) has positive rank $\bar{M}_i \neq 0$.  Observe that since $M = S^d$, the natural map $\bk^d \to \bar{M}$ is an isomorphism.  Moreover the action of $t$ on $\bar{M}$ is identified with the action of $\alpha_i$.  Now, $\bar{M}_i = \bar{M}/(t-x_i)\bar{M} = \bar{M}/(\alpha_i - 1)\bar{M} \neq 0$.  Hence $\alpha_i - 1$ is not invertible, $\alpha_i - 1$ has nonzero kernel, and $\alpha_i$ has a 1-eigenvector.

(2.) Let us compute $\chi(x_1,\dotsc,x_d,T)$.  We denote by $\delta^i_j$ the Kronecker function.  We have
\begin{align*} \chi_d(x_1,\dotsc,x_d,T) &= \sum_{\sigma \in S_d}{ (\sum_{i=1}^d{ x_i \alpha_i} - x_{\sigma(1)} ) \dotsm ( \sum_{i=1}^d{ x_i \alpha_i } - x_{\sigma(d)} ) } \\
& = \sum_{\sigma \in S_d}{ \prod_{j=1}^d\big(\sum_{i=1}^d{ x_i( \alpha_i - \delta^{i}_{\sigma(j)})} \big)} \\
& = \sum_{1 \leq i_1,\dotsc,i_d \leq d}{ x_{i_1}\dotsm x_{i_m}\big( \sum_{\sigma \in S_d}{(\alpha_{i_1} - \delta^{i_1}_{\sigma(1)})\dotsm (\alpha_{i_d} - \delta^{i_d}_{\sigma(d)}) } \big) }.
\end{align*}
In the second line the term order matters so the product is taken in the natural order $j = 1,2,\dotsc,d$.  Now suppose that $\chi_d(x_1,\dotsc,x_d,T) = 0$.  Then for all $1 \leq i_1, \dotsc,i_d \leq d$ we have
\begin{equation}\label{main-equation}
 \sum_{\sigma \in S_d}{ (\alpha_{i_1} - \delta^{i_1}_{\sigma(1)} )\dotsm (\alpha_{i_d} - \delta^{i_d}_{\sigma(d)}) } =0. 
\end{equation}

For each $j \neq i$, we consider the noncommutative monomial $x_i x_j x_i^{d-2}$ and its equation \eqref{main-equation},
\begin{equation}\label{ijequation} 
\sum_{\sigma \in S_d}{ (\alpha_i - \delta^i_{\sigma(1)})(\alpha_j - \delta^j_{\sigma(2)})(\alpha_i - \delta^i_{\sigma(3)})\dotsm(\alpha_i - \delta^i_{\sigma(d)})} = 0. 
\end{equation}
Note that since $\alpha_i(m) = m$, we calculate
\[ (\alpha_i - \delta^i_{\sigma(3)})\dotsm(\alpha_i - \delta^i_{\sigma(d)})m = 
\begin{cases} 
m & i \notin \{\sigma(3),\dotsc,\sigma(d)\}, \\ 
0 & i \in \{\sigma(3),\dotsc,\sigma(d)\}. 
\end{cases}
\]  Therefore applying \eqref{ijequation} to $m$ and simplifying we get
\begin{align*}
\sum_{\sigma \in S_d, \sigma(1)=i}{ (\alpha_i - 1)(\alpha_j - \delta^j_{\sigma(2)})m } + \sum_{\sigma \in S_d, \sigma(2) = i}{ \alpha_i \alpha_j m } &= (d-1)!\big( (\alpha_i - 1)(\alpha_j - \delta^j{\sigma(2)})m + \alpha_i \alpha_j m \big) \\
& = (d-1)!\big( (\alpha_i - 1)\alpha_j m + \alpha_i \alpha_j m \big) \\
& = (d-1)!(2\alpha_i - 1)\alpha_j m \\
& = 0,
\end{align*}
where passing from the first line to the second we use the fact that $(\alpha_i - 1)\delta^j_{\sigma(2)}m = 0$.  

Now, consider the special case of \eqref{main-equation} corresponding to $x_i^d$:
\[ \sum_{\sigma \in S_d}{ (\alpha_i - \delta^i_{\sigma(1)})\dotsm(\alpha_i - \delta^i_{\sigma(d)}) } = \sum_{j = 1}^d \sum_{\sigma\in S_d, \sigma(j) = i}{ \alpha_i^{j-1} (\alpha_i - 1) \alpha_i^{d-j-1} } = d! \alpha_i^{d-1}(\alpha_i - 1) = 0. \]
Since $\alpha_i^{d-1}(\alpha_i - 1) = 0$ it follows that $2\alpha_i - 1$ is invertible.  However, $(2\alpha_i - 1)\alpha_jm = 0$ so $\alpha_j m = 0$.
\end{proof}

\begin{proof}[Proof of Theorem \ref{thm:noncommutative}]
($\Leftarrow$)  We proceed by induction on $\dim(M)$ and fix an identification $M \cong \bk^n$.  Suppose $n=1$.  Then by Lemma \ref{lem:actions}, there is some $i$ and some $m \in \bk^1$ such that $\alpha_i(m) = m$.  Moreover, $\alpha_j m = 0$ for all $j \neq 0$.  Since $m$ spans $\bk^1$, the $\alpha_i$ satisfy the neccessary relations for $\phi$ to factor through an algebra morphism.

Now given $\alpha_1,\dotsc,\alpha_d \in M_n(\bk)$, Lemma \ref{lem:actions} implies that we can find an element $m \in \bk^d$ such that $\bk m \subset \bk^d$ is stable under the action of $\alpha_1,\dotsc,\alpha_d$.  Let $M_n(\bk,m) \subset M_n(\bk)$ be the algebra of operators that preserve $\bk m$.  Then there is a surjective algebra homomorphism $M_n(\bk,m) \onto M_{n-1}(\bk)$.  Since $\alpha_1,\dotsc,\alpha_d \in M_n(\bk,m)$ we find that $T \in M_n(\bk,m)\tensor_\bk \bk\langle x_1,\dotsc,x_d\rangle$.  So if $\alpha'_1,\dotsc,\alpha'_d \in M_{n-1}(\bk)$ are the images of $\alpha_1,\dotsc,\alpha_d$ then $T' = x_1 \alpha'_1 + \dotsc + x_d \alpha'_d$ satisfies $\chi_d$.  By induction we see that $(\alpha'_i)^2 = \alpha'_i$ and $\alpha'_i \alpha'_j = 0$ for $i \neq j$.  In particular, there is a codimension 1 subspace of $\bk^{n-1}$ preserved by $\alpha'_1,\dotsc,\alpha'_d$.  Its inverse image in $\bk^n$ (we identify $\bk^{n-1}$ with $\bk^n/\bk m$) is then a codimension one subspace $V' \subset \bk^d$ which is invariant under $\alpha_1,\dotsc,\alpha_d$.  Again by induction, $\alpha_i^2 - \alpha_i$ and $\alpha_i \alpha_j (i\neq j)$ annihilate $V'$.  There is some $i$ such that $\alpha_i$ acts by the identity on $\bk^n/V'$.  Since $\alpha_i^{d-1}(\alpha_i - 1) = 0$, the geometric multiplicity of 1 as an eigenvalue of $\alpha_i$ is equal to its algebraic multiplicity.  So there is a 1-eigenvector $m \in \bk^n$ whose image in $\bk^n / V'$ is nonzero.  Again Lemma \ref{lem:actions} implies that $\alpha_j m = 0$ for $j \neq i$.  Hence the relations $\alpha_i^2 - \alpha_i$ and $\alpha_i \alpha_j$ annihilate a basis for $\bk^n$ and hence annihilate $\bk^n$.

($\Rightarrow$)  Suppose that $\phi$ is an algebra map.  Then we have $\alpha_i^2 = \alpha_i$ for all $i$ and $\alpha_j\alpha_i = 0$ if $i \neq j$.  Decompose $\bk^n = V_1 \oplus \dotsc \oplus V_d$ where $V_i = \alpha_i(\bk^n)$.  Then $T$ preserves $V_i \tensor \bk\langle x_1,\dotsc,x_d\rangle$ for each $i$.  So we can view $T$ as an element of $\prod_{i=1}^d\End_\bk(V_i) \tensor \bk\langle x_1,\dotsc,x_d\rangle \subset M_n( \bk\langle x_1,\dotsc,x_d\rangle )$.  Since $(T - x_i)$ vanishes identically on $V_i \tensor \bk\langle x_1,\dotsc,x_d \rangle$ we see that for each $\sigma \in S_d$ and each $i$ the image of $(T - x_{\sigma(1)})\dotsm (T-x_{\sigma(d)})$ vanishes in $\End_\bk(V_i) \tensor \bk\langle x_1,\dotsc,x_d \rangle$ and hence in $M_n(\bk\langle x_1,\dotsc,x_d\rangle)$.  Since all of the terms of $\chi_d(T)$ vanish in $M_n(\bk\langle x_1,\dotsc,x_d \rangle)$, so does $\chi_d(T)$.
\end{proof}

\section*{Questions}
There are many natural questions that surround the notion of characteristic morphism.  We point out a few of them.

\begin{ques}
What are the irreducible characteristic morphisms for $A = \bk^{\times d}$?  Are there infinitely many for $d \geq 3$?
\end{ques}

Replacing a commutative semisimple algebra with a semisimple algebra, Theorem \ref{thm:roots} fails to hold.  Indeed, the map $\phi:\Mat_d(\bk) \to \Mat_d(\bk)$ defined by $\phi(M) = M^T$ is not a homomorphism, but does satisfy the hypotheses of Theorem \ref{thm:roots}.  Moreover, $\phi$ is a characteristic morphism.

\begin{ques}
Is there a characterization of when a linear map $\phi:\Mat_d(\bk) \to \Mat_r(\bk)$ is a homomorphism along the lines of Theorem \ref{thm:roots}? 
\end{ques}

Let $V$ is a finite dimensional vector space and $F(t) \in \Sym^\bt(V^\vee)[t]$ be monic and homogeneous.  Given $v \in V$ we can consider the image $F_v(t)$ of $F(t)$ under the homomorphism $\Sym^\bt(V^\vee)[t] \to \bk[t]$ induced by $v:V^\vee \to \bk$.  The main theorem of \cite{CK} implies that there always exists a linear map $\phi:V \to \Mat_r(\bk)$ for some $r$ such that $F_v(\phi(v)) =0$ for all $v \in V$.  There is a natural non-commutative generalization of this problem.  

\begin{ques}
For which monic, homogenous elements $F(t)$ of $\rmT(V^\vee) \ast \bk[t]$, does there exist an element $\phi^\vee \in V^\vee \tensor \Mat_r(V)$ for some $r$ such that $F(\phi^\vee) = 0$ in $\rmT(V^\vee) \tensor \Mat_r(\bk)$?
\end{ques}

If $F(t)$ is the symmetrization of the characteristic polynomial of an algebra structure on $V$ then we have an affirmative answer.  However, if $F(t) = t^2 - u \tensor v$ where $u,v$ are linearly independent, then there is no such element.

\bibliographystyle{alpha}
\bibliography{cpr.bib}

\end{document}